\documentclass[11pt,a4paper]{amsart}
\usepackage{a4}
\usepackage{amssymb,latexsym}
\def\eqnreset{\setcounter{equation}{0}}
\def\eqsection#1{\section{#1}\eqnreset}

\newtheorem{Thm}{Theorem}[section]
\newtheorem{Defi}[Thm]{Definition}
\newtheorem{Cor}[Thm]{Corollary}
\newtheorem{Lemma}[Thm]{Lemma}
\newtheorem{Prop}[Thm]{Proposition}
\newtheorem{Rem}[Thm]{Remark}
\newtheorem{Conj}[Thm]{Conjecture}
\newtheorem{Prelim}[Thm]{Preliminary}

\newenvironment{thm}[0]{\begin{Thm}\noindent}%
{\end{Thm}}
\newenvironment{defi}[0]{\begin{Defi}\noindent\rm}%
{\end{Defi}}
\newenvironment{cor}[0]{\begin{Cor}\noindent}%
{\end{Cor}}
\newenvironment{lemma}[0]{\begin{Lemma}\noindent}%
{\end{Lemma}}
{\end{Prop}}
{\end{Rem}}
{\end{Conj}}
{\end{Prelim}}

\def\qed{~\hfill$\square$\medbreak}

\def\text#1{\;\;\;\;{\rm \hbox{#1}}\;\;\;\;}
\def\qquad{\quad\quad}

\def\msy#1{{\mathbb #1}}
\def\C{{\msy C}}
\def\N{{\msy N}}
\def\Z{{\msy Z}}

\def\gl{\lambda}

\def\gL{\Lambda}

\def\frak#1{\mathfrak #1}
\def\fa{{\frak a}}

\def\fg{{\frak g}}

\def\fk{{\frak k}}

\def\fn{{\frak n}}
\def\fp{{\frak p}}
\def\fq{{\frak q}}

\def\fu{{\frak u}}

\def\to{\rightarrow}
\def\Re{{\rm Re}\,}

\def\Ad{{\rm Ad}}

\def\Hom{{\rm Hom}}

\def\cA{{\mathcal A}}

\def\cH{{\mathcal H}}

\def\cU{{\mathcal U}}

\def\rmS{{\rm S}}
\def\rmU{{\rm U}}

\def\PW{{\rm PW}}
\def\Exp{{\rm Exp}}
\def\supp{\mathrm s\mathrm u\mathrm p\mathrm p}

\newcommand{\Ee}{\mathcal{V}}
\newcommand{\Eea}{\mathcal{V}^a}
\newcommand{\Cr}{\mathrm{Cr}(G/K_0)}
\def\sideremark#1{\ifvmode\leavevmode\fi\vadjust{\vbox to0pt{\vss
 \hbox to 0pt{\hskip\hsize\hskip1em%
 \vbox{\hsize2cm\tiny\raggedright\pretolerance10000
 \noindent #1\hfill}\hss}\vbox to8pt{\vfil}\vss}}}

\begin{document}
\setcounter{section}{0}
\title[Fourier series on compact symmetric spaces]{Fourier series on compact symmetric spaces: K-finite
functions of small support}

\author{Gestur \'Olafsson}
\thanks{Research of \'Olafsson was supported by NSF grants DMS-0402068
and  DMS-0801010}
\address{Department of Mathematics, Louisiana State University, Baton Rouge,
LA 70803, U.S.A.}
\email{olafsson@math.lsu.edu}
\author{Henrik Schlichtkrull}
\address{Department of Mathematics, University of Copenhagen, Universitetsparken 5,
DK-2100 K{\o}benhavn {\O}, Denmark}
\email{schlicht@math.ku.dk}


\begin{abstract}
The Fourier coefficients of a function $f$
on a compact symmetric space
$U/K$ are given by integration of $f$ against
matrix coefficients of irreducible representations of $U$.
The coefficients depend on a spectral parameter $\mu$,
which determines the representation, and they can be
represented by elements $\widehat f(\mu)$
in a common Hilbert space $\cH$.

We obtain a theorem of Paley-Wiener type which
describes the size of the support of $f$
by means of the exponential type
of a holomorphic $\cH$-valued extension of $\widehat f$,
provided $f$ is $K$-finite and of sufficiently small support.
The result was obtained previously
for $K$-invariant functions, to which case we reduce.
\end{abstract}
\maketitle

\eqsection{Introduction.}
\noindent
The present paper is a continuation of our article  \cite{OS}.
We consider a Riemannian symmetric space $X$ of compact type,
realized as the homogeneous space $U/K$ of a compact Lie group $U$.
Up to covering, $U$ is the connected component of the group
of isometries of $X$. As an example,
we mention the sphere $S^n$, for which $U={\rm SO}(n+1)$
and $K={\rm SO}(n)$. In the cited paper, we considered
$K$-invariant functions on $U/K$. The Fourier series of
a $K$-invariant function $f$ is
\begin{equation}
\label{e: K-invariant Fourier series}
\sum_{\mu}a_\mu\psi_\mu(x),
\end{equation}
where $\psi_\mu$
is the zonal spherical function
associated with the representation of $U$ with highest weight
$\mu$, and where the Fourier coefficients $a_\mu$ are given by
\begin{equation}\label{eq-sphericalFT}
a_\mu= d(\mu)\tilde f(\mu)=
d(\mu)\int_{U/K} f(x)\overline{\psi_\mu(x)}\,dx,
\end{equation}
with $d(\mu)$ being the representation dimension,
and $dx$ being the normalized invariant  measure on $U/K$.
The main result of \cite{OS} is a local Paley-Wiener
theorem, which gives a necessary and sufficient condition
on the coefficients in the series
(\ref{e: K-invariant Fourier series})
that it is the Fourier series of a smooth
$K$-invariant function $f$
supported in a geodesic ball of a given
sufficiently small radius $r$ around the
origin in $U/K$. The condition is, that $\mu \mapsto a_\mu$
extends to a holomorphic function of
exponential type $r$ satisfying certain invariance under
the action of
the Weyl group.
We refer to \cite{BOP,BOP2,C06,G01} for previous results on special cases.
The case of distributions will be treated in \cite{OS08b}.

In the present paper we consider the general case where
the $K$-invariance is replaced by
$K$-finiteness. Instead of being scalars, the Fourier
coefficients take values in the Hilbert space $\cH=L^2(K/M)$,
where $M$ is a certain subgroup of $K$.
In case of $U/K=S^n$, we have $K/M=S^{n-1}$.
Our main result is Theorem \ref{t: main} below, which describes
the set of Fourier coefficients of $K$-finite smooth functions
on $U/K$, supported in a ball of a given sufficiently small radius.
The corresponding result for Riemannian
symmetric spaces of the non-compact type is due
to Helgason, see \cite{Sig73}.

Our method is by reduction to the $K$-invariant case.
For the reduction we
use Kostant's description of the spherical principal series
of a semisimple Lie group \cite{Kostant}. A similar reduction
was found by Torasso \cite{Torasso} for Riemannian symmetric
spaces of the non-compact type, thus providing an alternative
proof of the mentioned theorem of Helgason.

\eqsection{Basic notation}
\label{s: notation}\noindent
We recall some basic notation from \cite{OS}.
We are considering a Riemannian symmetric space $U/K$, where
$U$ is a connected compact semisimple Lie group and
$K$ a closed symmetric subgroup. By definition this means
that there exists a nontrivial  involution $\theta$ of $U$ such that
$K_0\subset K\subset U^\theta$. Here $U^\theta$ denotes the
subgroup of $\theta$-fixed points, and
$K_0:=U^\theta_0$ its identity component. The
base point in $U/K$ is denoted by $x_0=eK$.

The Lie algebra of $U$ is denoted $\fu$, and by
$\fu=\fk\oplus\fq$ we denote the
Cartan decomposition associated with the involution $\theta$.
We endow $U/K$ with the Riemannian structure induced
by the negative of the Killing form on $\fq$.

Let $\fa\subset\fq$ be a maximal abelian subspace,
$\fa^*$ its dual space, and $\fa^*_\C$ the complexified
dual space. The set of non-zero weights for $\fa$ in $\fu_\C$
is denoted by $\Sigma$. The roots $\alpha\in\Sigma\subset\fa^*_\C$
are purely imaginary valued on $\fa$.
The corresponding Weyl group, generated by the reflections
in the roots, is denoted $W$. We make a fixed choice of
a positive system $\Sigma^+$ for $\Sigma$, and define
$\rho\in i\fa^*$ to be half the sum of the roots in $\Sigma^+$,
counted with multiplicities.
The centralizer of $\fa$ in $K$ is denoted $M=Z_K(\fa)$.

Some care has to be taken because we are not assuming $K$
is connected. We recall that if $U$ is simply connected,
then $U^\theta$ is connected and $K=K_0$, see \cite{Sig}, p. 320.
We recall also that in general $K=MK_0$, see \cite{OS}, Lemma 5.2.

In the following we shall need to complexify $U$ and $U/K$.
Since $U$ is compact there exists a unique (up to isomorphism)
connected complex Lie group $U_\C$ with Lie algebra $\fu_\C$
which contains $U$ as a real Lie subgroup.
Let $\fg$ denote the real form $\fk+i\fq$ of $\fu_\C$,
and let $G$ denote the connected real Lie subgroup of $U_\C$
with this Lie algebra.
Then $\fg_\C=\fu_\C$ as complex vector
spaces, and $U_\C$ complexifies $G$ as well as $U$. In particular,
the almost complex structures that $\fu$ and $\fg$ induce
on $U_\C$ are identical. For this reason we shall
denote $U_\C$ also by $G_\C$.
The Cartan involutions of $\fu$ and $U$
extend to involutions of $\fg_\C$ of $G_\C$, which we shall denote
again by $\theta$, and which leave $\fg$ and $G$ invariant.
The corresponding Cartan decomposition of $\fg$ is
$\fg=\fk+i\fq$. It follows that $K_0=G^\theta$ is maximal compact in $G$,
and $G/K_0$ is a Riemannian symmetric space of the non-compact type.

We denote by $\fg=\fk\oplus i\fa\oplus \fn$ and
$G=K_0AN$ the Iwasawa decompositions of $\fg$ and $G$
associated with $\Sigma^+$. Here $A=\exp(i\fa)$ and
$N=\exp\fn$.
Furthermore, we
let $H\colon G\to i\fa$ denote the {\it Iwasawa projection}
$$K_0AN\ni k\exp Y n=g\mapsto H(g)=Y.$$

Let $K_{0\C}$, $A_\C$, and $N_\C$ denote the connected subgroups
of $G_\C$ with Lie algebras $\fk_\C$, $\fa_\C$ and $\fn_\C$,
and put $K_\C=K_{0\C}K$. Then  $G_\C/K_\C$
is a symmetric space, and it
carries a natural complex structure with respect to which
$U/K$ and $G/K_0$ are totally real submanifolds
of maximal dimension.

\begin{lemma}\label{l: complex Iwasawa map}
There exists an open $K_\C\times K$-invariant
neighborhood $\Eea$ of the neutral element
$e$ in $G_\C$, and a holomorphic map
\begin{equation}\label{e: complex H}
H\colon \Eea\to\fa_\C,
\end{equation}
which agrees with the Iwasawa projection on $\Eea\cap G$,
such that
\begin{equation}\label{e: Iwa}
u\in K_\C\exp(H(u))N_\C
\end{equation}
for all $u\in\Eea$.
\end{lemma}

\begin{proof}(See \cite{clerc} or \cite{Stanton}.)
We first assume that $K=U^\theta$. Then $K_\C=G_\C^\theta$.
Since $\fg_\C=\fn_\C\oplus\fa_\C\oplus\fk_\C$,
there exist an open
neighborhood $T_{\fn_\C}\times T_{\fa_\C}$ of $(0,0)$
in $\fn_\C\times\fa_\C$ such that
the map
$$
\fn_\C\times\fa_\C\ni(X,Y)\mapsto
\exp X\exp Y\cdot x_0\in G_\C/K_\C
$$
is a biholomorphic diffeomorphism of
$T_{\fn_\C}\times T_{\fa_\C}$
onto an open neighborhood $\Ee$ of
$x_0=eK_\C$ in $G_\C/K_\C$. We assume, as we may, that
$T_{\fn_\C}$ and $T_{\fa_\C}$ are invariant under the
complex conjugation with respect to the real form $\fg$.

We denote by $\Eea$
the open set $\{x\mid x^{-1} K_\C\in \Ee\}\subset G_\C$.
The map
$$K_\C\times T_{\fa_\C}\times T_{\fn_\C}\ni(k,Y,X)\mapsto
k\exp Y\exp X\in \Eea\subset G_\C$$
is then a biholomorphic diffeomorphism.

In particular, the map
$H\colon \Eea\to\fa_\C$ defined by
$$k\exp Y\exp X\mapsto Y$$
for $k\in K_\C$, $Y\in T_{\fa_\C}$ and $X\in T_{\fn_\C}$,
is holomorphic and satisfies (\ref{e: Iwa}).

The conjugation with respect to $\fg$ lifts to an
involution of $G_\C$ that leaves $G$ pointwise fixed.
Moreover, since this conjugation commutes with $\theta$,
it stabilizes $K_\C$. Hence it stabilizes $\Eea$.
Let  $u\in\Eea\cap G$
and write $u=k\exp Y\exp X$ with
$k\in K_\C$, $Y\in T_{\fa_\C}$ and $X\in T_{\fn_\C}$.
It follows that
$k$, $Y$ and $X$ are fixed by the conjugation.
In particular, $Y\in i\fa$ and $X\in \fn$, and
hence $k=u\exp(-X)\exp(-Y)\in G\cap K_\C=K_0$.
Therefore, $u=k\exp Y\exp X$ is the Iwasawa decomposition,
and $H(u)=Y$ the Iwasawa projection, of $u$.

We postpone the condition of right-$K$-invariance and
consider the general case where $K_0\subset K\subset
U^\theta$. We retain the
sets $T_{\fn_\C}$ and $T_{\fa_\C}$
from above and recall that $K_\C=K_{0\C}K$ is an open subgroup
of the previous $K_\C$.
Again we define
$\Eea=K_{\C}\exp(T_{\fa_\C})\exp(T_{\fn_\C})$.
This is an open subset of the previous $\Eea$.
The restriction of the previous $H$ to this set
is obviously holomorphic, agrees with Iwasawa on $\Eea\cap G$,
and it is easily seen to satisfy (\ref{e: Iwa}).

Finally, we note that $\Eea$ contains an $\Ad K$
invariant open neighborhood $V$ of $e$ in $G_\C$. Hence,
for each $k\in K$, the set $\Eea k$ is left-$K_\C$-invariant
and contains $V$.
The intersection $\cap_{k\in K} \Eea k$ is $K_\C\times K$
invariant and contains $V$. The interior of this set has
all the properties requested of $\Eea$.
\qed\end{proof}

We call the map in (\ref{e: complex H}) the
{\it complexified Iwasawa projection}.
A particular set $\Eea$ as above can be constructed as follows.
Let
$$\Omega =\{X\in \fa \mid
( \forall \alpha\in\Sigma)\,\, |\alpha (X)|<\pi/2\}.$$
The set
$$\Ee=\Cr=G\,\exp \Omega\, K_\C\subset G_\C/K_\C,$$
called the {\it complex crown} of $G/K_0$,
was introduced in
\cite{AG90}.
Its preimage in $G_\C$
is open and contained in
$N_\C A_\C K_\C\subset G_\C$.
This is shown for all classical groups in \cite{KrSt}, Theorem 1.8,
and in general in \cite{H02}, Theorem 3.21. See also \cite{GM},
\cite{M03}.
Let $\Eea=\{x^{-1}\mid x\in \Ee\}\subset G_\C$.
The existence of the holomorphic Iwasawa projection $\Eea\to\fa_\C$ is
established in \cite{KrSt}, Theorem 1.8, with a proof that
can be repeated in the general case.
It follows that $\Eea$
has all the properties mentioned in Lemma \ref{l: complex Iwasawa
  map}.

One important property of the crown is
that it is $G$-invariant and that all the spherical functions on
$G/K$ extends to a holomorphic function on the crown (it is in fact maximal with this
property, see \cite{kos}, Theorem 5.1).
However, this property plays no role in the present article,
where we shall just assume that $\Eea$ has the properties in
Lemma \ref{l: complex Iwasawa map}, and $\Ee=(\Eea)^{-1}$.

\eqsection{Fourier analysis}
\label{s: Fourier analysis}\noindent
In this section
we develop a local Fourier theory for
$U/K$ based on elementary representation
theory. The theory essentially
originates from Sherman \cite{ShermanBull}.

An irreducible unitary representation $\pi$ of $U$ is
said to be {\it spherical} if there exists a non-zero
$K$-fixed vector $e_\pi$ in the representation space $V_\pi$.
The vector $e_\pi$ (if it exists) is unique up to multiplication
by scalars. After normalization to unit length we obtain
the matrix coefficient
$$\psi_\pi(u)=\langle \pi(u)e_\pi,e_\pi\rangle$$
which is the corresponding \textit{zonal spherical function}.

{}From the point of view of
representation theory it is natural to define the
Fourier transform of an integrable function $f$ on
$U/K$ to be the map that associates the vector
$$\pi(f) e_\pi
= \int_{U} f(u\cdot x_0) \pi(u)e_{\pi} \, du
=\int_{U/K} f(x) \pi(x)e_{\pi} \, dx\in V_\pi,$$
to each spherical representation, with a fixed choice of
the unit vector $e_\pi$ for each $\pi$
(see \cite{OS08a} for discussion on the noncompact case).
The corresponding Fourier series is
\begin{equation}
\label{Fourier series}
\sum_\pi d(\pi)\,
\langle \pi(f)e_\pi ,\pi(x)e_\pi \rangle
\end{equation}
for $x\in U/K$.
It converges to $f$ in $L^2$ if $f$ belongs to $L^2(U/K)$, and it
converges uniformly if $f$ has a sufficient number of continuous
derivatives (see \cite{taylor}).

In the case of the sphere $S^2$, the expansion of $f$
in spherical harmonics $Y^m_l(x)$ (with integral labels $|m|\leq l$)
is obtained from this expression when we
express $\pi(x)e_\pi$ by means of
an orthonormal basis for the
$(2l+1)$-dimensional representation space $V_\pi=V_l$.

For the purpose of Fourier analysis it is convenient
to embed all the representation spaces $V_\pi$,
where $\pi$ is spherical, in a common
Hilbert space $\cH$, independent of $\pi$, such that
$\widehat f$ can be viewed as an $\cH$-valued function on the set of
equivalence classes of irreducible spherical representations.
This can be achieved as follows.

Recall that in the classification of Helgason,
a spherical representation $\pi=\pi_\mu$
is labeled by an element $\mu\in \fa^*_\C$,
which is the restriction, from a compatible maximal torus,
of the highest weight of $\pi$ (see \cite{GGA}, p. 538).
We denote by $\gL^+(U/K)\subset\fa^*_\C$ the set of
these restricted highest weights, so that $\mu\mapsto\pi_\mu$
sets up a bijection from $\gL^+(U/K)$ onto the set of
equivalence classes of irreducible $K$-spherical
representations.
According to the theorem of
Helgason, every $\mu\in\Lambda^+(U/K)$
satisfies
\begin{equation}
\label{e: Helgason condition}
\frac{\langle\mu,\alpha\rangle}{\langle\alpha,\alpha\rangle}
\in\Z^+,
\end{equation}
for all $\alpha\in\Sigma^+$, where the brackets
denote the inner product induced by the Killing form. Furthermore,
if $U$ is simply connected, then an element
$\mu\in\fa^*_\C$ belongs to $\gL^+(U/K)$ if
and only if it satisfies (\ref{e: Helgason condition}).
For the description in the general case, one must
supplement (\ref{e: Helgason condition}) by
both the assumption that $\pi_\mu$ descends to $U$,
and that the $K_0$-fixed vector is also $K$-fixed.

For each $\mu\in\gL^+(U/K)$ we fix an irreducible unitary
spherical representation
$(\pi_\mu,V_\mu)$ of $U$
and  a unit $K$-fixed vector $e_\mu\in V_\mu$.
Furthermore, we fix a  highest weight vector $v_\mu$
of weight $\mu$, such that $\langle v_\mu,e_\mu\rangle=1$.
The following lemma is proved in \cite{GGA}, p. 535,
in the case that $U$ is simply connected.

\begin{lemma}
\label{l: K-span}
Let $\mu\in\gL^+(U/K)$. Then
$\pi_\mu(m)v_\mu=v_\mu$ for all $m\in M$, and the
vectors
$\pi_\mu(k) v_\mu$, where $k\in K_0$, span the space $V_\mu$.
\end{lemma}

\begin{proof} Let $m\in M$ be given.
Since $m$ centralizes $\fa$ and normalizes
$\fn$, it follows that $\pi_\mu(m)v_\mu$ is again
a highest weight vector of the same weight.
Hence $\pi_\mu(m)v_\mu=c v_\mu$. By taking
inner products with $e_\mu$, which is $M$-fixed,
it follows that $c=1$.
The statement about the span follows directly from the
Iwasawa decomposition $G=K_0AN$.\qed
\end{proof}

It follows from Lemma \ref{l: K-span} that
the map $V_\mu \to L^2(K/M)$, $v\mapsto \langle v , \pi_\mu(\cdot )v_\mu
\rangle$, is injective.
We shall use the space $\cH=L^2(K/M)$ as our common model
for the spherical representations. It will be convenient
to use an anti-linear embedding of $V_\mu$. Hence we
define for $\mu\in\gL^+(U/K)$
\begin{equation}\label{eq-hv}
h_v(k)=\langle\pi_\mu(k)v_\mu,v\rangle,\quad (k\in K)
\end{equation}
and $\cH_\mu=\{h_v\mid v\in V_\mu\}$.
Then $v\mapsto h_v$ is a $K$-equivariant anti-isomorphism
$V_\mu\to\cH_\mu\subset\cH$.

Notice that $h_{e_\mu}=1$, the constant function on $K/M$.
Hence $1$ belongs to $\cH_\mu$ for all $\mu\in\gL^+(U/K)$.
Although we shall not use it in the sequel, we also note
that every $K$-finite function in $\cH=L^2(K/M)$ belongs to
$\cH_\mu$ for some $\mu$ (this can be seen from results
explained below, notably Lemma \ref{l: embedding} and
equation (\ref{e: Kostant for psi}), where for a given $K$-type
$\delta$ one chooses $\mu$ such that $P(-\mu-\rho)$
is non-singular).

According to the chosen embedding of $V_\mu$ in $\cH$, we define
the {\it Fourier transform} of an integrable function $f$
on $U/K$ by $$\tilde f(\mu)=
\int_{U/K} f(u)\, h_{\pi_\mu(u)e_\mu} \,du\in\cH$$
for $\mu\in\gL^+(U/K)$, that is
\begin{equation}%{eqnarray}
\tilde f(\mu,b)=
%&=&\int_{U/K} f(u)\, \overline{ \langle\pi_\mu(u)e_\mu ,\pi_\mu(k)v_\mu\rangle} \,du\nonumber\\&=&
\int_{U/K} f(u)\, \langle\pi_\mu(k)v_\mu ,\pi_\mu(u)e_\mu\rangle \,du,
\label{e: Fourier transform}
\end{equation}%{eqnarray}
for $b=kM\in K/M$.
If $f$ is $K$-invariant, then $\widetilde{f} (\mu) $ is
independent of~$b$. Integration
over $K$ then shows that this definition
agrees with the spherical Fourier transform in (\ref{eq-sphericalFT}).

It is easily seen that the Fourier transform $f\mapsto \widetilde{f}(\mu )$
is intertwining for the left regular actions of  $K$ on $U/K$
and $K/M$, respectively. In particular, it maps $K$-finite
functions on $U/K$ to $K$-finite functions on $K/M$.

We now invoke the complex group $G_\C$
and the complexified Iwasawa projection defined
in the preceding section. Let $\Eea\subset G_\C$
and $H\colon\Eea\to\fa_\C$ be as in Lemma
\ref{l: complex Iwasawa map}, and let $\mu\in\gL^+(U/K)$.
Since $\pi_\mu$ extends to a holomorphic representation
of $G_\C$, it follows from Lemma \ref{l: complex Iwasawa map} that
$\langle\pi_\mu(u)v_\mu,e_\mu\rangle=e^{\mu(H(u))}$
for all $u\in \Eea$.
Let $\Ee=\{x^{-1}\mid x\in \Eea\}\subset G_\C$. Then
\begin{equation}\label{e: Iwasawa expression}
\langle \pi_\mu(k)v_\mu, \pi_\mu(u)e_\mu\rangle
=e^{\mu(H(u^{-1}k))}
\end{equation}
for $k\in K$, $u\in U\cap\Ee$ and $\mu\in\Lambda^+(U/K)$.

\begin{lemma}
\label{l: holo ext}
Let $f$ be an integrable function on $U/K$ with
support in $U\cap\Ee$.
Then
\begin{equation}
\label{e: Fourier with H}
\tilde f(\mu,k)=
\int_{U/K} f(u)\, e^{\mu(H(u^{-1}k))} \,du,
\end{equation}
for all $k\in K/M$,
and the Fourier transform $\mu\mapsto \tilde f(\mu)$
extends to a holomorphic $\cH$-valued function on
$\fa^*_\C$, also denoted by $\tilde f$,
satisfying the same equation {\rm (\ref{e: Fourier with H})}.
Moreover,
\begin{equation}
\label{e: pi(f) inversion}
\pi_\mu(f)e_\mu=\int_{K/M} \tilde f(-\mu-2\rho,k)
\pi_\mu(k)v_\mu\, dk
\end{equation}
for all $\mu\in\Lambda^+({U/K})$.
\end{lemma}

The measure on $K/M$ used in (\ref{e: pi(f) inversion})
is the quotient of the normalized Haar measures on $K$ and $M$.
\smallskip

\begin{proof} The expression (\ref{e: Fourier with H}) follows
immediately from (\ref{e: Fourier transform})
and (\ref{e: Iwasawa expression}). The integrand in
(\ref{e: Fourier with H}) depends holomorphically on $\mu$,
locally uniformly with respect to $u$ and $k$.
Hence an analytic continuation is
defined by this formula.

In order to establish the identity (\ref{e: pi(f) inversion})
it suffices to show that
$$\pi_\mu(u)e_\mu=\int_{K/M} e^{-(\mu+2\rho)H(u^{-1}k)}
\pi_\mu(k)v_\mu \,dk$$
for $u\in U\cap\Ee$. The latter identity is easily shown
to hold for $u\in G$ (use \cite{GGA}, p. 197, Lemma 5.19,
and the fact that $K/M=K_0/(M\cap K_0)$).
By analytic continuation it then holds for $u\in \Ee_0$,
the identity component of $\Ee$. Since $\Ee=\Ee_0K_\C$,
it follows for all $u\in\Ee$.
\qed
\end{proof}

\begin{Cor}
\label{c: Sherman}
{\rm (Sherman)}
Assume $f\in L^2(U/K)$ has
support
contained in $U\cap\Ee$.
Then the sum
$$
\sum_{\mu\in\Lambda^+(U/K)} d(\mu)
\int_{K/M} \tilde f(-\mu-2\rho,k)  \,
\langle\pi_\mu(k)v_\mu,\pi_\mu(x)e_\mu \rangle \,dk,\quad x\in U/K,
$$
converges to $f$ in $L^2(U/K)$, and it
converges uniformly if $f$ has a sufficient number of continuous
derivatives.
\end{Cor}

\begin{proof}
(See \cite{ShermanBull}).
Follows immediately
from (\ref{Fourier series}) by insertion of
(\ref{e: pi(f) inversion}).\qed
\end{proof}

In \cite{ShermanActa} the inversion formula of
Corollary \ref{c: Sherman} is
extended to a formula for
functions on $U/K$ without
restriction on the support
(for symmetric spaces of rank one).
We shall not use this extension here.
For the special case of the sphere $U/K=S^n$,
see also \cite{ShermanTAMS}, \cite{Strichartz}
and \cite{Y}.

\eqsection{The spherical principal series}\noindent
The space $\cH=L^2(K/M)=L^2(K_0/(M\cap K_0))$
is the representation space for
the spherical principal series for $G$.
We denote by $\sigma_\lambda$ this series of
representations, given by
\begin{equation}
\label{e: spherical principal series}
[\sigma_\lambda(g)\psi](k)=e^{-(\lambda+\rho)H(g^{-1}k)}
 \psi(\kappa(g^{-1}k))
\end{equation}
for $\lambda\in\fa_\C^*$, $g\in G$, $\psi\in\cH$ and $k\in K_0$.
Here $\kappa\colon G\to K_0$ is the Iwasawa projection
$kan\mapsto k$.

Let $\mu\in\gL^+(U/K)$.
By extending $\pi_\mu$ to a holomorphic
representation of $G_\C$ and then restricting to $G$, we obtain
a finite dimensional
representation of $G$, which we again denote by~$\pi_\mu$.
We now have the following well-known result.
It relates the embedding of $V_\mu$ into $\cH$,
which  motivated (\ref{e: Fourier transform}),
to the principal series representations.

\begin{lemma}
\label{l: embedding}
Let $\mu\in\gL^+(U/K)$.  The map $v\mapsto h_v$ defined by
{\rm (\ref{eq-hv})} provides a
$G$-equivariant embedding of the contragredient of
$\pi_\mu$ into $\sigma_{-\mu-\rho}$.
\end{lemma}

\begin{proof}
Recall that the contragredient representation
can be realized on the conjugate Hilbert space $\bar V_\mu$
by the operators $\pi_\mu(g^{-1})^*$, and
notice that $v\mapsto h_v$  is {\it linear} from $\bar V_\mu$
to $\cH$.
Since $v_\mu$ is
a highest weight vector it follows easily from
(\ref{e: spherical principal series})  that
$$\sigma_{-\mu-\rho}(g)h_v=h_{\pi_\mu(g^{-1})^*v}$$
for $g\in G$.
\qed
\end{proof}

The space $C^\infty(K/M)\subset \cH$ carries
the family of representations, also denoted by $\sigma_\lambda$,
of $\fg_\C$ obtained by differentiation and complexification.
Thus, although the representations $\sigma_\lambda$ of $G$
in general do not complexify to global representations of $U$,
the infinitesimal representations $\sigma_\lambda$ of $\fu$
are defined for all $\lambda\in\fa^*_\C$. We denote by
$\cH^\infty_\lambda$ the space
$C^\infty(K/M)$ equipped with the representation $\sigma_\gl$
of $\fu_\C=\fg_\C$, and with the left regular representation of
$K$.

\begin{lemma}
\label{l: u-homo}
The Fourier transform $f\mapsto \tilde f(\mu)$ defines a
$(\fu,K)$-homomorphism
{}from $C^\infty(U/K)$ to $\cH^\infty_{-\mu-\rho}$,
for all $\mu\in\gL^+(U/K)$. Moreover, the holomorphic extension,
defined in Lemma \ref{l: holo ext},
restricts to a $(\fu,K)$-homomorphism
from
$$\{f\in C^\infty(U/K)\mid \supp f\subset U\cap \Ee\}$$
to $\cH^\infty_{-\mu-\rho}$ for all $\mu\in\fa^*_\C$.
\end{lemma}

\begin{proof} Since $\pi_\mu$ is a unitary representation of $U$
it follows from  Lemma \ref{l: embedding} that
$\sigma_{-\mu-\rho}(X)h_v=h_{\pi_\mu(X)v}$ for $X\in\fu$, $v\in V_\mu$.
The first statement now follows, since $$\tilde f(\mu)=\int_{U/K}
f(u) h_{\pi_\mu(u)e_\mu}\,du.$$

It follows from Lemma \ref{l: uniqueness} and Theorem \ref{t: into}
below, that the second statement can be derived from the first by analytic
continuation with respect to $\mu$, provided the support of $f$
is sufficiently small. However, we prefer to give an independent
proof, which only requires assumptions on the support of $f$
as stated in the lemma.

Since the Fourier transform in (\ref{e: Fourier with H}) is clearly
$K$-equivariant,
it suffices to prove the intertwining property
\begin{equation}\label{eq-intertwiningProperty}
[L(X)f]^\sim(\mu)=\sigma_{-\mu-\rho}(X)\tilde f(\mu)
\end{equation}
for $X\in\fq$.  By definition
$$[L(X)f](u)=\frac{d}{dt}\Big|_{t=0} f(\exp(-tX)u)$$
and hence by invariance of the measure
$$[L(X)f]^\sim(\mu,k)=
\int_{U/K} f(u) \,\frac{d}{dt}\Big|_{t=0}\,
e^{\mu(H(u^{-1}\exp(-tX)k))}\,du.$$

Let $\fp=i\fq$ so that $\fg=\fk+\fp$ is the Cartan
decomposition of $\fg$, and write $X=iY$ for $Y\in\fp$.
Since the complexified Iwasawa map $H$ is holomorphic, it
follows that
$$\frac{d}{dt}\Big|_{t=0}\,
e^{\mu(H(u^{-1}\exp(-tX)k))}=i\frac{d}{dt}\Big|_{t=0}\,
e^{\mu(H(u^{-1}\exp(-tY)k))}.$$
Furthermore
$$H(u^{-1}\exp(-tY)k)=H(u^{-1}\kappa(\exp(-tY)k))+H(\exp(-tY)k)
$$
and hence we derive
\begin{eqnarray}
&&[L(X)f]^\sim(\mu,k)\nonumber\\
&&\qquad=
i\frac{d}{dt}\Big|_{t=0} \left[ e^{\mu(H(\exp(-tY)k))}
\int_{U/K} f(u)
e^{\mu(H(u^{-1}\kappa(\exp(-tY)k))}\,du\right]
\nonumber\\
&&\qquad=
i\frac{d}{dt}\Big|_{t=0} \left[e^{\mu(H(\exp(-tY)k))}
\tilde f(\mu,\kappa(\exp(-tY)k))\right].\nonumber\\
&&\qquad=
i\frac{d}{dt}\Big|_{t=0} \left[\sigma_{-\mu-\rho}(\exp(tY))
\tilde f(\mu)\right](k).\nonumber
\end{eqnarray}
Since by definition $\sigma_{-\mu-\rho}(X)=i\sigma_{-\mu-\rho}(Y)$,
the last expression is exactly
$\sigma_{-\mu-\rho}(X)\tilde f(\mu)$ evaluated at $k$.
\qed\end{proof}

We recall that there exist normalized standard intertwining
operators between the principal series:
$$\cA(w,\lambda)\colon \cH\to\cH, \quad w\in W\, ,$$
such that
\begin{equation}\label{e: intertwining property of A}
\sigma_{w\lambda}(g)\circ\cA(w,\lambda)=
\cA(w,\lambda)\circ\sigma_\lambda(g)
\end{equation}
for all $g\in G$.
The normalization is such that
\begin{equation}\label{e: normalization of A}
\cA(w,\lambda)1=1
\end{equation}
for the constant function $1$ on $K/M$. The map
$\lambda\mapsto \cA(w,\lambda)$ is meromorphic with values
in the space of bounded linear operators on $\cH$.

We need the following property of the {\it Poisson kernel},
which is defined for $x\in G$ and $k\in K_0$ by
$e^{-(\lambda+\rho)H(x^{-1}k)}$. By Lemma
\ref{e: complex H} it is defined also for $x\in\Ee$ and $k\in K$.

\begin{lemma}
\label{l: A on Poisson}
The identity
\begin{equation}
\label{e: A on e}
\cA(w,\lambda)e^{-(\lambda+\rho)H(x^{-1}\cdot)}
=e^{-(w\lambda+\rho)H(x^{-1}\cdot)},
\end{equation}
of functions in $\cH$, holds for all $x\in\Ee$.
\end{lemma}

\begin{proof}
The identity is well-known for $x\in G$. In fact in this case
it follows easily from (\ref{e: spherical principal series}),
(\ref{e: intertwining property of A}) and (\ref{e: normalization of A}).
The map
$x\mapsto e^{\mu (H(x^{-1}\cdot))}$ is holomorphic
$\cH$-valued on $\Ee$ for each $\mu\in\fa^*_\C$,
because the complexified Iwasawa projection
is holomorphic. Hence
(\ref{e: A on e}) holds
for $x\in \Ee_0$ by analytic continuation,
and then for $x\in\Ee$ by the obvious
left-$K_\C$-invariance
of both sides with respect to $x^{-1}$.
\qed\end{proof}

\eqsection{The $K$-finite Paley-Wiener space}

\noindent
For each irreducible representation $\delta$  of $K_0$ we denote
by $\cH_\delta$ the finite dimensional subspace of $\cH$
consisting of the functions that generate
an isotypical representation of type $\delta$. Likewise, for
each finite set $F$ of $K_0$-types, we denote by $\cH_F$ the sum
of the spaces $\cH_\delta$ for $\delta\in F$.
Obviously, the intertwining operators $\cA(w,\lambda)$ preserve
each subspace $\cH_F$. Although we do not need it in the sequel,
we remark that
$\lambda\mapsto \cA(w,\lambda)|_{\cH_F}$ is a rational map
from $\fa_\C^*$ into the space of linear operators on the finite
dimensional space $\cH_F$, for each $F$, see \cite{W72}.

Note that since $K/K_0$ is finite, a function on $K/M=K_0/(K_0\cap M)$
is $K_0$-finite if and only if it is $K$-finite. We use the
notations $\cH_\delta$ and $\cH_F$ also for an irreducible
representation $\delta$ of $K$, and for a set $F$ of $K$-types.

\begin{defi}
\label{d: PW space}
For $r>0$ the {\it $K$-finite Paley-Wiener space}
$\PW_{K,r}(\fa)$ is
the space of holomorphic functions $\varphi$ on $\fa_\C^*$
with values in $\cH=L^2(K/M)$
satisfying the following.

\item{(a)} There exists a finite set $F$ of $K$-types such that
$\varphi(\gl)\in \cH_F$ for all $\gl\in\fa_\C^*$.

\item{(b)} For each $k\in\N$ there exists a constant $C_k>0$ such that
$$\|\varphi(\gl)\|\leq C_k(1+|\gl|)^{-k} e^{r|\Re\gl|}$$
for all $\gl\in\fa_\C^*.$

\item{(c)} The identity
$\varphi(w(\mu+\rho)-\rho)=\cA(w,-\mu-\rho)\varphi(\mu)$
holds for all $w\in W$, and for generic $\mu\in\fa_\C^*$.
\end{defi}

We note that the norm on $\fa^*_\C$ used in (b) is induced by
the negative of the Killing form on $\fa$.
In particular we see that $\PW_{K,r}(\fa)=\PW_{K_0,r}(\fa)$,
that is, the $K$-finite Paley-Wiener space is the same for all the
spaces $U/K$ where $K_0\subset K\subset U^\theta$.

Notice that the Paley-Wiener space $\PW_r(\fa)$ defined in \cite{OS}
can be identified with the space of functions $\varphi$
in $\PW_{K,r}(\fa)$, for which $\varphi(\lambda)$
is a constant function on $K/M$ for each $\lambda$.
This follows from the normalization
(\ref{e: normalization of A}).

The functions in the Paley-Wiener space are uniquely determined by their
restriction to $\gL^+(U/K)$, at least when $r$ is sufficiently small.
This is seen in the following lemma.

\begin{lemma}
\label{l: uniqueness}
There exists $R>0$ such that if
$\varphi\in\PW_{K,r}(\fa)$ for some $r<R$ and
$\varphi(\mu)=0$ for all $\mu\in\gL^+(U/K)$, then $ \varphi=0$.
\end{lemma}

\begin{proof}  The relevant value of $R$ is
the same as in \cite{OS} Thm. 4.2 (iii) and Remark 4.3.
The lemma follows easily from application of
\cite{OS}, Section 7, to the function
$\gl\mapsto \langle\varphi(\gl,\,\cdot\,), \psi\rangle$
for each $\psi\in\cH$.\qed
\end{proof}

Obviously $\PW_{K,r}(\fa)$ is $K$-invariant, where
$K$ acts by the left regular representation
on functions on $K/M$. The following lemma shows that it
is also a $(\fu,K)$-module.

\begin{lemma}
\label{l: (u,K)}
Let $r>0$, $\varphi\in\PW_{K,r}(\fa)$ and $X\in\fu_\C$. Then
the function $\psi=\sigma(X)\varphi$ defined by
$$\psi(\lambda)=\sigma_{-\lambda-\rho}(X)(\varphi(\lambda))\in\cH$$
for each $\lambda\in\fa^*_\C$, belongs to $\PW_{K,r}(\fa)$.
\end{lemma}

\begin{proof}
Recall that $\sigma_{-\gl-\rho}(X)$ is defined
by complexification of the infinitesimal action of
$\fg$ on the smooth functions in $\cH$, and note that
$\varphi(\lambda)$ is smooth on $K/M$, since it is $K$-finite.
Hence we may assume $X\in\fg$.
It is easily seen that $\psi(\lambda)$
is $K_0$-finite, of types which occur in the tensor product of
the adjoint representation $\mathrm{Ad}$ of $K_0$ on $\fg$
with types from $F$. Hence condition (a) is valid for
the function $\psi$.
Condition (c) follows immediately from the intertwining
property of $\cA(w,\lambda)$. It remains to verify
holomorphicity in $\lambda$, and the estimate
in (b) for $\psi$.

By definition both the holomorphicity
and norm in the estimate (b)
refer to the Hilbert space $\cH=L^2(K/M)$. However,
because of condition (a)
and since $\cH_F$ is finite dimensional, it is equivalent
to require holomorphicity of $\psi(\gl)(x)$ pointwise for
each $x\in K/M$, and likewise to require the exponential
estimate for $\psi(\gl)(x)$ pointwise with respect to $x$.
Thus let an element $x=kM\in K/M$ be fixed, where $k\in K_0$.

Note that by (\ref{e: spherical principal series})
\begin{equation}%
(\sigma(X)\varphi)(\lambda)(k)=%&=&
\frac{d}{dt}\Big|_{t=0}\, e^{-(\lambda+\rho)(H(\exp(-tX)k))}\,
\varphi(\lambda)(\kappa(\exp(-tX)k)).
\nonumber%\\&=&
%\frac{d}{dt}\Big|_{t=0}\, e^{-(\lambda+\rho)(H(\exp(-tX^k))}\,
%\varphi(\lambda)(\kappa(\exp(-tX)k)).\nonumber
\end{equation}
Differentiating with the Leibniz rule, we obtain a sum of
two terms.

The first term is
\begin{equation}
\label{e: first term}
\frac{d}{dt}\Big|_{t=0} \left( e^{-(\lambda+\rho)(H(\exp(-tX)k))} \right)
\varphi(\lambda)(k).
\end{equation}
Let $\alpha(Z)=H(\exp(Z)k)\in i\fa$ for $Z\in\fg$, then $\alpha(0)=0$ and
it follows that (\ref{e: first term}) equals
$$
(\lambda+\rho)(d\alpha_0(X))\,\varphi(\lambda)(k)
$$
where $d\alpha_0$ is the differential of $\alpha$ at $0$.
It is now obvious that (\ref{e: first term})
is holomorphic and satisfies the same the growth
estimate as $\varphi(\lambda)(k)$.
Hence (b) is valid for the
first term.

The second term is
\begin{equation}
\label{e: second term}
\frac{d}{dt}\Big|_{t=0}
\varphi(\lambda)(\kappa(\exp(-tX)k)),
\end{equation}
which we rewrite as follows.
Let
$$\beta(Z)=\kappa(\exp(Z)k)k^{-1}\in K_0$$
for $Z\in\fg$, then $\beta(0)=e$ and
$$\varphi(\lambda)(\kappa(\exp(-tX)k))=\varphi(\gl)(\beta(-tX)k).$$
It follows that (\ref{e: second term}) equals
$$
L(d\beta_0(X))(\varphi(\lambda))(k)
$$
where $d\beta_0(X)\in T_eK_0=\fk$.
The linear operator $L(d\beta_0(X))$ preserves
the finite dimensional space $\cH_F$ and hence
restricts to a bounded linear operator on that space.
It follows that (\ref{e: second term})
is holomorphic in $\gl$ and satisfies~(b).
\qed\end{proof}

\eqsection{Fourier transform maps into Paley-Wiener space}
\noindent
In this section we prove the following result. Let $C_K^\infty(U/K)$
denote the space of $K$-finite smooth functions on $U/K$, and for
each $r>0$ let
$$C^\infty_{K,r}(U/K)=\{f\in C_K^\infty(U/K)\mid
\supp f \subset\Exp(\bar B_r(0))\}$$
where $\bar B_r(0)$ denotes the closed ball in $\fq$
of radius $r$ and center
$0$, and $\Exp$ denotes the exponential map of $U/K$.

\begin{thm}
\label{t: into}
There exists a number $R>0$ such that
$\Exp(\bar B_R(0))\subset U\cap\Ee$ and such that
the following holds
for every $r<R$:

If $f\in C^\infty_{K,r}(U/K)$,
then the holomorphic extension of $\tilde f$ from
Lemma \ref{l: holo ext} belongs to
$\PW_{K,r}(\fa)$.
\end{thm}

In the proof we shall reduce to the case where $K=K_0$.
The following lemma prepares the way for this reduction.

The projection $p: U/K_0\to U/K$ is a covering map. Hence we
can choose $R>0$ such that $p$ restricts to a
diffeomorphism of the open ball $\Exp(B_R(0))$ in $U/K_0$
onto the open ball $\Exp(B_R(0))$ in $U/K$.
It follows that for each $r<R$
a bijection $F\mapsto f$ of $C^\infty_{K_0,r}(U/K_0)$ onto
$C^\infty_{K,r}(U/K)$ is defined by
$$f(u)=\sum_{v\in K/K_0} F(uv),\quad u\in U$$
for $F\in C^\infty_{K_0,r}(U/K_0)$,
where for each $u$ at most one term is non-zero.
The inverse map is given by
$$F(u)= \begin{cases} f(p(u)), & u\in \Exp(B_R(0)),\\0,
&{\rm otherwise,}
\end{cases}$$
for $f\in C^\infty_{K,r}(U/K)$. Let $\Eea\subset G_\C$ be as in
Lemma \ref{l: complex Iwasawa map}, and note that this set
also satisfies the assumptions of that lemma for the
symmetric space $U/K_0$. As before, let $\Ee=\{x^{-1}|x\in\Eea\}$.

\begin{lemma}
\label{l: K to K0} Let $f\in C^\infty_{K,r}(U/K)$ and
$F\in C^\infty_{K_0,r}(U/K_0)$ be as above.
Then $f$ is supported in $U\cap\Ee$ if and only
$F$ is supported in $U\cap\Ee$. In this case, the analytically continued
Fourier transforms of these functions satisfy
$$\tilde f(\mu)=c\tilde F(\mu)$$
for all $\mu\in\fa^*_\C$, where $c$ is the index of $K_0$ in $K$.
\end{lemma}

\begin{proof} It follows from the definition of the map $F\mapsto f$ that
$$\tilde f(\mu,k)=\int_U \sum_{v\in K/K_0} F(uv)e^{\mu(H(u^{-1}k))}\,du
=c\tilde F(\mu,k)
$$
by right-$K$-invariance of the
Haar measure and left-$K$-invariance of $H$.\qed
\end{proof}

We can now give the proof of Theorem \ref{t: into}.

\smallskip

\begin{proof}
Property (a) in Definition \ref{d: PW space}
follows immediately
from the fact that the Fourier transform is $K$-equivariant.
Moreover, the transformation law for the
Weyl group in Property (c) follows easily
from Lemma \ref{l: A on Poisson} by integration
over $U\cap\Ee$ against $f(u)$.

For the proof of Property (b), with $r$ bounded by a suitable value $R$,
we reduce to the case that $K$ is connected.
We assume that $R$ is sufficiently small as
described above Lemma \ref{l: K to K0}.
Then according to the lemma, given a function $f\in  C^\infty_{K,r}(U/K)$,
the function  $F\in C^\infty_{K_0,r}(U/K_0)$ has the same
Fourier transform up to a constant. The reduction now follows
since $\PW_{K,r}(\fa)=\PW_{K_0,r}(\fa)$,
as mentioned below Definition \ref{d: PW space}.
For the rest of this proof we assume $K=K_0$.

It is known from \cite{OS}, Thm. 4.2(i),
that the estimate in Property (b)
holds for $K$-invariant functions on $U/K$.
We prove the property in general by reduction to that case.
In particular, we can
use the same value of $R>0$ (see \cite{OS}, Remark 4.3).

Fix an irreducible $K$-representation $(\delta,V_\delta)$.
It suffices to prove the result for functions $f$ that
transform isotypically under $K$ according to this type.

We shall use Kostant's description in \cite{Kostant}
of the $K$-types
in the spherical principal series. We draw the results we
need directly from the exposition in
\cite{GASS}, Chapter 3.
In particular, we denote by $H^*_\delta$ the
finite dimensional subspace of
the enveloping algebra $\cU(\fg)$ which is the image
under symmetrization  of the space of harmonic polynomials
on $\fp$ of type $\delta$, and we denote by
$E_\delta$ the space
$$E_\delta=\Hom_K(V_\delta,H^*_\delta),$$
of linear $K$-intertwining maps
$V_{\delta}\to H^*_\delta$.
It is known that $E_\delta$ has the same dimension as $V^M_\delta$.

We denote by $\Hom^*(V_\delta^M,E_\delta)$ the space
of {\it anti-linear} maps $V_\delta^M\to E_\delta$.
The principal result we need is Theorem 2.12 of \cite{GASS}, p. 250,
according to which there exists
a  rational function
$P=P^\delta$ on $\fa_\C^*$ with values
in $\Hom^*(V_\delta^M,E_\delta)$
such that
\begin{equation}
\label{e: Kostant}
\int_{K/M} e^{-(\lambda+\rho)H(x^{-1}k)}
\langle v,\delta(k)v'\rangle\,dk
=
[L(P(\lambda)(v')(v))\varphi_\lambda](x)
\end{equation}
for all $v\in V_\delta$, $v'\in V_{\delta}^M$ and $x\in G/K$,
and for $\lambda\in\fa_\C^*$ away
from the singularities of $P(\lambda)$.
Here $L$ denotes the action of the enveloping algebra
from the left on functions on $G/K$, and
$\varphi_\lambda$ denotes the spherical function
$$
\varphi_\lambda(x)=
\int_{K/M} e^{-(\lambda+\rho)H(x^{-1}k)}\,dk$$
on $G/K$.

The equality
(\ref{e: Kostant}) is valid for $x\in U\cap\Ee$
by analytic continuation. Let $f\in C^\infty_{K,r}(U/K)_\delta$,
where $r<R$ and the subscript $\delta$ indicates that $f$
is $K$-finite of this type.
Then
$$f(x)=d(\delta)\int_K \chi_\delta(l) f(lx)\,dl$$
for all $x\in U$, where $\chi_\delta$ is the character of $\delta$.
It follows that
$$
\tilde f(\mu,k)= d(\delta)
\int_{U/K}\int_K \chi_\delta(l) f(lu) \, dl \,\,e^{\mu(H(u^{-1}k))} \,du
$$
and hence by Fubini and invariance of measures
$$\tilde f(\mu,k)= d(\delta)
\int_{U/K}\int_{K/M}\int_M \chi_\delta(lmk^{-1}) \,dm\,
e^{\mu(H(u^{-1}l))} \, dl\, f(u)  \,du.
$$
The inner expression
$\int_M \chi_\delta(lmk^{-1})\,dm$
is a finite sum of matrix coefficients of the form
$\langle \delta(l)v,\delta(k)v'\rangle$ with $v\in V_\delta$ and
$v'\in V^M_{\delta}$, and hence it follows from (\ref{e: Kostant})
that $\tilde f(\mu,k)$ for generic $\mu\in\fa^*_\C$
is a finite sum of expressions of the form
$$\int_{U/K} [L(P(-\mu-\rho)(\delta(k)v')(v))
\varphi_{-\mu-\rho}](u) f(u)\,du$$
with $v$ and $v'$ independent of $\mu$ and $k$.
In these expressions the right invariant differential operators
$L(P(-\mu-\rho)(\delta(k)v')(v)$ can be thrown over, by taking
adjoints.
Since the spherical function is $K$-invariant, we finally obtain
\begin{equation}
\label{e: Lf}
\int_{U/K}
\varphi_{-\mu-\rho}(u)
\int_K [L(P(-\mu-\rho)(\delta(k)v')(v))^*f](yu)\,dy \,du.
\end{equation}

Notice that
(\ref{e: Lf}) is the spherical Fourier transform
from \cite{OS}, Section~6.
It follows that $\tilde f(\mu,k)$, for $\mu$ generic and $k\in K$,
is a finite sum in which
each term has the form of the spherical Fourier transform
applied to the $K$-integral of a derivative of $f$ by a differential
operator with coefficients that depends rationally on $\mu$
and continuously on $k$.
The application of
a differential operator to $f$ does not increase the support,
hence it follows from the estimates in \cite{OS} that each
term is a rational multiple of a function of $\mu$ of
exponential type, with estimates which are
uniform with respect to $k$. It then follows from
\cite{GASS} Lemma 5.13, p. 288, and its proof,
that the Fourier transform
$\tilde f(\mu, k)$ itself is of the same exponential type.
We have established Property (b) in Definition \ref{d: PW space}
for $\tilde f$.
\qed
\end{proof}

\eqsection{Fourier transform maps onto Paley-Wiener space}
\noindent
Let $\varphi\in \PW_{K,r}(\fa)$ for some $r>0$ and consider
the function $f$ on $U/K$ defined by the Fourier series
\begin{equation}
\label{e: inverse PW}
f(x)=
\sum_{\mu\in\Lambda^+(U/K)} d(\mu)
\int_{K/M} \varphi(-\mu-2\rho,k)  \,
\langle\pi_\mu(k)v_\mu,\pi_\mu(x)e_\mu \rangle \,dk.
\end{equation}
It follows from the estimate in Property (b) of
Definition \ref{d: PW space} that the sum converges
and defines a smooth function on $U/K$ (see \cite{Sugiura}).

\begin{thm}
\label{t: onto}
There exists a number $R>0$ such that
$\Exp(\bar B_R(0))\subset U\cap\Ee$ and such that
the following holds
for every $r<R$.
For each $\varphi\in \PW_{K,r}(\fa)$
the function $f$ on $U/K$ defined by (\ref{e: inverse PW})
belongs to
$C^\infty_{K,r}(U/K)$ and has Fourier transform
$\tilde f=\varphi$.
\end{thm}

\begin{proof}
Again we first reduce to the case that $K$ is connected.
Assuming that the theorem is valid in that case, we find
a number $R>0$ such that
every function $\varphi\in \PW_{K_0,r}(\fa)$, where $r<R$,
is of the form $\tilde F$ for some $F\in  C^\infty_{K_0,r}(U/K_0)$.
We may assume that $R$ is as small as explained above
Lemma~\ref{l: K to K0}. Let $\varphi\in \PW_{K,r}(\fa)$ be given
and recall that $\PW_{K,r}(\fa)=\PW_{K_0,r}(\fa)$. Let
$F\in  C^\infty_{K_0,r}(U/K_0)$ with $\tilde F=c^{-1}\varphi$,
and construct $f\in  C^\infty_{K,r}(U/K)$ as in Lemma
\ref{l: K to K0}. It follows from the lemma that
$\tilde f=c\tilde F=\varphi$,
and then it follows from Corollary \ref{c: Sherman}
that $f$ is the  function given by (\ref{e: inverse PW}).
This completes the reduction.

For the rest of this proof, we assume that $K=K_0$.
The value of $R$ that we shall use
is the same as in \cite{OS}, Thm. 4.2(ii) and Remark 4.3.
We may assume that $\varphi(\gl,\cdot)$ is isotypical of a given
$K$-type $\delta$ for all $\gl\in\fa^*_\C$.

For $v\in V_\delta$ and $v'\in V^M_\delta$
we denote by $\psi_{v,v'}$
the matrix coefficient
$$\psi_{v,v'}(k)=\langle v,\delta(k)v'\rangle$$
on $K/M$. By the Frobenius reciprocity
theorem it follows that
these functions $\psi_{v,v'}$ span the space
$\cH_\delta$. Moreover, it follows from the definition of
the standard intertwining operators
by means of integrals over quotients of
$\theta(N)$, that these operators
act on each function $\psi_{v,v'}$
only through the second variable. That is,
there exists a linear map
$$B(w,\lambda)\colon V^M_\delta \to V^M_\delta$$
such that
\begin{equation}
\label{e: A on psi}
\cA(w,\lambda)\psi_{v,v'}=\psi_{v,B(w,\lambda)v'}.
\end{equation}
for all $v,v'$.
Notice that the dependence of $B(w,\lambda)$ on $\lambda$
is anti-meromorphic.

It follows (by using a basis for $V_\delta$)
that we can write $\varphi(\mu,k)$
as a finite sum of functions of the form
$$\psi_{v,v'(\mu)}(k)$$
where $v\in V_\delta$ is fixed and where
$v'\colon \fa^*_\C\to V_\delta^M$ is anti-holomorphic
of exponential type $r$ and satisfies the
transformation relation in Definition \ref{d: PW space} (c), that is,
\begin{equation}
\label{e: W on v'}
v'(w(\mu+\rho)-\rho)=B(w,-\mu-\rho)v'(\mu)
\end{equation}
for $w\in W$.

Since the Poisson transformation for $G/K$ is equivariant
for the left action and injective for generic $\lambda$,
it follows from (\ref{e: Kostant}), by applying the inverse
Poisson
transform on both sides, that
\begin{equation}
\label{e: Kostant for psi}
\psi_{v,v'}=
\sigma_\lambda(P(\lambda)(v')(v))1
\end{equation}
for all $v\in V_\delta$, $v'\in V_\delta^M$
(see also \cite{GASS}, Thm. 3.1, p. 251), and
for all $\lambda$ for which
$P(\lambda)$ is non-singular. Here $1$ denotes the
constant function with value 1 on $K/M$.
We apply (\ref{e: Kostant for psi}) for $\lambda=-\mu-\rho$
generic
to the function
$\psi_{v,v'(\mu)}$
and thus obtain
our Paley-Wiener function $\varphi(\mu,\cdot)$
as a finite sum of
elements of the form
$$
\sigma_{-\mu-\rho}(P(-\mu-\rho)(v'(\mu))(v))1.
$$

The functions $P\colon \fa^*_\C \to \Hom^*(V^M_\gl,E_\delta)$
satisfy the following transformation property
\begin{equation}
\label{e: W on P}
P(w\lambda)\circ B(w,\lambda)=P(\lambda).
\end{equation}
Indeed, it follows from
(\ref{e: Kostant for psi}), (\ref{e: A on psi})
and (\ref{e: intertwining property of A})
that
$$\sigma_{w\lambda}(P(w\lambda)(B(w,\lambda)v')(v))1
=\sigma_{w\lambda}(P(\lambda)(v')(v))1$$
for all $v$ and $v'$, and generic $\lambda$. The identity
(\ref{e: W on P}) follows, since the map $u\mapsto \sigma_\nu(u)1$
is injective from $H^*_\delta$ to $\cH$ for generic $\nu$
according to \cite{GASS}, Thm. 3.1, p. 251
(alternatively, (\ref{e: W on P}) follows from
\cite{GASS}, Thm. 3.5, p. 254).

It follows from (\ref{e: W on P}) combined with (\ref{e: W on v'})
that the function
$$\mu\mapsto u(\mu):=P(-\mu-\rho)(v'(\mu))(v)\in
H^*_\delta$$ satisfies
$u(w(\mu+\rho)-\rho)=u(\mu)$ for generic  $\mu$, that
is, the shifted function $\lambda\mapsto u(\lambda-\rho)$
is $W$-invariant. Notice that $u$ is a rational multiple of
a holomorphic function of $\mu$,
since $P(-\mu-\rho)$ is antilinear in $v'$, and $v'$
is antiholomorphic in~$\mu$.

It follows from \cite{GASS}, Prop. 4.1, p. 264, that
$\lambda\mapsto P(-\lambda)$ is non-singular
on an open neighborhood of the set where
$$\Re\langle\lambda,\alpha\rangle\geq 0$$
for all roots $\alpha\in\Sigma^+$. Hence
$u(\lambda-\rho)$ is holomorphic on this set.
By the above-mentioned $W$-invariance
the function is then  holomorphic everywhere. Since it is a
rational multiple of a function of exponential type $r$,
we conclude from \cite{GASS}, Lemma 5.13, p. 288,
that it has exponential type $r$.

Since $H^*_\delta$ is finite dimensional
we thus obtain an expression for
$\varphi(\lambda,\cdot)$ as a finite sum
of functions of the form
$$\varphi_i(\lambda) \sigma_\lambda(u_i)1,$$
with scalar valued functions $\varphi_i$ on $\fa^*_\C$
which are $W$-invariant
(for the action twisted by $\rho$)
and of exponential type $r$,
and with $u_i\in H^*_\delta$.

According to the theorem proved in \cite{OS}, each function
$\varphi_i$ is the spherical Fourier transform of a
$K$-invariant smooth function $f_i\in C^\infty_r(U/K)$.
The function $L(u_i)f_i$ also belongs to
$C^\infty_r(U/K)$, and by Lemma \ref{l: u-homo}
it has Fourier transform $\varphi_i(\lambda)\sigma_\lambda(u_i)1$.
We conclude that if $f$ is the sum of the $L(u_i)f_i$, then
$\tilde f=\varphi$, as desired.

Finally, it follows from Corollary \ref{c: Sherman}
that $f$ is identical to
the function given by the Fourier series (\ref{e: inverse PW}).
\qed\end{proof}

We combine Theorems \ref{t: into} and \ref{t: onto} to obtain
the following.

\begin{thm}
\label{t: main}
There exists a number $R>0$ such that the Fourier transform
is a bijection of $C^\infty_{K,r}(U/K)$ onto
$\PW_{K,r}(\fa)$ for all $r<R$.
\end{thm}

We note the following corollary, which is analogous to a
result of Torasso in the non-compact case (see
\cite{GASS}, Cor. 5.19, p. 291).

\begin{cor}
There exists $r>0$ such that each function in
$C^\infty_{K,r}(U/K)$ is a finite linear combination of
derivatives of $K$-invariant
functions in $C^\infty_r(U/K)$ by members
of $\,\mathcal U(\fg)$, acting from the left.
\end{cor}

\begin{proof} More precisely, the proof above shows that
if $f\in C^\infty_{K,r}(U/K)$ is $K$-finite
of isotype $\delta$, then $f=\sum_i L(u_i)f_i$
with $u_i\in H^*_\delta$ and $f_i\in C^\infty_r(U/K)^K$.
\end{proof}

\eqsection{Final remarks}\noindent
Every function $f\in C^\infty(U/K)$ can be expanded in
a sum of $K$-types,
\begin{equation}
\label{e: K-type expansion}
f=\sum_{\delta\in\widehat K} f_\delta
\end{equation}
where $f_\delta\in C^\infty_\delta(U/K)$
is obtained from $f$ by left convolution
with the character of $\delta$ (suitably normalized).
It is easily seen that $f$ is supported in a given
closed geodesic ball $B$ around $x_0$, if and only if
each $f_\delta$ is supported in $B$.
The following is then a consequence of Theorem
\ref{t: main}.

\begin{cor}\label{c: non-K-finite}
There exists $R>0$ with the following
property. Let $f\in C^\infty_R(U/K)$ and $r<R$.
Then $f\in C^\infty_r(U/K)$ if and only if
the Fourier transform $\tilde f_\delta$ of
each of the functions $f_\delta$ allows a
holomorphic continuation satisfying the
growth estimate (b) of Definition \ref{d: PW space}
(with constants depending on $\delta$).
\end{cor}

For example, in the case of the sphere $S^2$, the expansion
(\ref{e: K-type expansion}) of $f$ reads $f=\sum_{m\in\mathbb Z} f_m$,
and the Fourier transform of $f_m$ is the map
\begin{equation}
\label{e: spherical harmonic}
l\mapsto \left\{\begin{array}{cl} c_{l,m} & \mathrm{for}\, \, l\geq |m|\cr
0&\mathrm{for}\, \,  0\leq l<|m|
\end{array}
\right.
\end{equation}
where $c_{m,l}$ are the coefficients of
the spherical harmonics expansion
$$f=\sum_{l=0}^\infty (2l+1)\sum_{|m|\leq l} c_{l,m}Y^m_l.$$
The condition in Corollary
\ref{c: non-K-finite} is thus that the map
(\ref{e: spherical harmonic}) has a holomorphic extension
to $l\in\mathbb C$ of the proper exponential type, for each $m\in\mathbb Z$.

It is an obvious question, whether the assumption of
$K$-finiteness can be removed in Theorem \ref{t: main}.
It is not difficult to remove it from Theorem \ref{t: onto}.
Assume that $\varphi$ satisfies Properties (b) and (c)
in Definition \ref{d: PW space} for a suitably
small value of $r$. Define a function $f:U/K\to \mathbb{C}$
by (\ref{e: inverse PW}). Using the arguments from \cite{Sugiura,taylor}
it follows that  $f\in C^\infty(U/K)$. By expanding
$f$ as in (\ref{e: K-type expansion})
it follows from Corollary \ref{c: non-K-finite}
that $f$ has support inside the ball of radius $r$.
It also follows that $\tilde f=\varphi$.

The nontrivial part would be to remove the assumption
from Theorem \ref{t: into}.  At this point we do not know if the Fourier
transform actually maps all non-$K$-finite functions of small support
into the space of functions satisfying the estimate in
Property (b). The ingredients in our proof, in particular
the matrices $P(\lambda )$, depend on the $K$-types. We would like
to point out that for the noncompact dual $G/K$, this direction
is proved in \cite{GASS}, p. 278, using the Radon transform.
It has been suggested to us by Simon Gindikin that \cite{Simon}
might be used in such an argument for $U/K$.

\end{document}